\documentclass[review,onefignum,onetabnum]{siamonline190516}
\usepackage{amssymb}
\usepackage{amsmath}
\usepackage{amsfonts}

\usepackage{lipsum}
\usepackage{mathrsfs}
\usepackage{amsfonts}
\usepackage{graphicx}
\usepackage{epstopdf}
\usepackage{algorithmic}
\ifpdf
  \DeclareGraphicsExtensions{.eps,.pdf,.png,.jpg}
\else
  \DeclareGraphicsExtensions{.eps}
\fi

\usepackage{enumitem}
\setlist[enumerate]{leftmargin=.5in}
\setlist[itemize]{leftmargin=.5in}

\newsiamremark{remark}{Remark}
\newsiamremark{hypothesis}{Hypothesis}
\crefname{hypothesis}{Hypothesis}{Hypotheses}
\newsiamthm{claim}{Claim}

\title{The monotonicity method for the inverse elastic scattering on unbounded domains \thanks{Submitted to the editors DATE.
}
}

\author{
Bastian Harrach\thanks{Institute for Mathematics, Goethe-University Frankfurt, Frankfurt am Main, Germany (\email{harrach@math.uni-frankfurt.de}).} \and
Jianli Xiang \thanks{Corresponding author: Three Gorges Mathematical Research Center, College of Mathematics and Physics, China Three Gorges University, Yichang 443002, China (\email{xiangjianli@ctgu.edu.cn}).}
}

\usepackage{amsopn}

\ifpdf
\hypersetup{
  pdftitle={Inverse medium scattering problem},
  pdfauthor={Jianli Xiang}
}
\fi

\nolinenumbers

\begin{document}

\maketitle

\begin{abstract}
  We discuss a time-harmonic inverse scattering problem for the Navier equation with compactly supported penetrable and possibly inhomogeneous scattering objects in an unbounded homogeneous background medium, and we develop a monotonicity relation for the far field operator that maps superpositions of incident plane waves to the far field patterns of the corresponding scattered waves. Combining the monotonicity relation with the method of localized potentials, we extend the so called monotonicity method to characterize the support of inhomogeneities in the Lam\'{e} parameters and the density in terms of the far field operator.
\end{abstract}



\section{Introduction}

The wave scattering problem is an important research direction in the inverse problem of partial differential equations, which has been widely used in engineering fields such as nondestructive testing, environmental science, geophysical exploration and medical diagnosis. While the well-posedness of the direct scattering problem has been thoroughly investigated through the integral equation and variational methods, the inverse problem has also attracted a wide variety of extensive and intensive investigations \cite{Colton2019}.

The reconstruction of the position and shape of unknown scatterers from the far field data is a fundamental but severely ill-posed problem in inverse scattering problems. In the past two decades, efficient qualitative reconstruction algorithms have received widespread attention, and there are two representative non-iterative methods: decomposition methods and sampling methods. Decomposition methods include the dual space method \cite{Colton2019} and the point source method \cite{Potthast2001}, and sampling methods include the singular sources method \cite{Potthast2000}, the probe method \cite{Ikehata1998}, the linear sampling method \cite{Colton1996} and the factorization method \cite{Kirsch2008}, whose main idea is to construct a certification associated with measurement data to detect the targeted object. Among qualitative methods for shape reconstruction, the monotonicity method has been recently introduced by Harrach in \cite{Harrach2013} for the electrical impedance tomography. It is formulated in terms of far field operators that map superpositions of incident plane waves, which are being scattered at the unknown scattering objects, to the far field patterns of the corresponding scattered waves. Comparing with the factorization method \cite{Kirsch2008}, the general theorem of the monotonicity method does not assume that the real part of the middle operator of the far field operator has a decomposition into a positive coercive operator and a compact operator, which means that the monotonicity method generates reconstruction schemes under weaker a priori assumptions for unknown targets \cite{Furuya2021}.

In \cite{Lakshtanov2016}, Lakshtanov and Lechleiter have generalized the factorization method for inverse medium scattering using a particular factorization of the difference of two far field operators and obtained a monotonicity principle which yields a simple algorithm to
compute upper and lower bounds for boundary values. Therefore, the monotonicity method is closely related to the factorization method. Very recently, the monotonicity analysis from \cite{Harrach2013} has been extended to inverse coefficient problems for the Helmholtz equation in a bounded domain for fixed nonresonance frequency and real-valued scattering coefficient function \cite{Harrach2019,HarrachP2019}, where the authors have shown a monotonicity relation between the scattering coefficient and the local Neumann-to-Dirichlet operator. Combining this with the method of localized potentials, they have derived a constructive monotonicity based
characterization of scatterers from partial boundary data \cite{Harrach2019} and improved the bounds for the space dimension \cite{HarrachP2019}. Then, Griesmaier and Harrach \cite{Griesmaier2018} have made a generalization of these results to the inverse medium scattering problem on unbounded domains with plane wave incident fields and far field observations of the scattered waves. Furthermore, the monotonicity method has also been extended to the inverse mixed obstacle scattering \cite{Albicker2020}, an inverse Dirichlet crack detection \cite{Daimon2020}, an open periodic waveguide \cite{Furuya2020}, a closed cylindrical waveguide \cite{Arens2023} and the references therein \cite{Albicker2023,Furuya2021,Harrach2023}.

Concerning the isotropic linear elasticity in the stationary case, the monotonicity result between the Lam\'e parameters and the Neumann-to-Dirichlet operator and the existence of localized potentials has been presented in \cite{Eberle2020}, which has been applied to detect and reconstruct inclusions based on the standard as well as linearized monotonicity tests in \cite{Eberle2021,Eberle2023,EberleM2021}. To make a significant improvement over standard regularization techniques, Eberle and Harrach have dealed with the same problem by the monotonicity-based regularization method \cite{Eberle2022}. For the non-stationary or time harmonic case of the Navier equation, the paper \cite{Eberle2024} has extended the monotonicity method for inclusion detection and shown how to determine certain types of inhomogeneities in the Lam\'e parameters and the density. The main contribution of the present work is the generalization of the monotonicity method to the time-harmonic inverse elastic scattering problem on unbounded domains. Our approach relies on the monotonicity of the far field operator with respect to the Lam\'e parameters as well as the density and the techniques of localized potentials.

The outline of this article is as follows. After briefly introducing the mathematical setting of the scattering problem in Section 2, we develop the monotonicity relation for the far field operator in Section 3. In Section 4 we discuss the existence of localized wave functions for the Navier equation in unbounded domains, and we use them to provide a converse of the monotonicity relation from Section 3. In Section 5 we establish rigorous characterizations of the support of scattering objects in terms of the far field operator.

\section{Problem formulation}

In this paper, we consider the inverse medium scattering problem of time-harmonic elastic waves and deal with the shape reconstruction problem, which is also known as the detection problem. Assume that the propagation of time-harmonic waves is in an isotropic non-absorbing and inhomogeneous elastic medium with the density function $\rho$ and Lam\'{e} constants $\mu$ and $\lambda$ satisfying $\mu>0$, $\mu+\lambda>0$ in $\mathbb{R}^{2}$. We are specifically interested in determining the region where the material properties $\lambda$, $\mu$ or $\rho$ differ from some known constant background values $\lambda_{0}$, $\mu_{0}$ or $\rho_{0}$, when given a set of far field measurements in the form of the far field operator. This problem corresponds physically to determining the inhomogeneous regions in the body $\Omega$ from far field measurements. Here we assume that $\lambda_{0}$, $\mu_{0}$, $\rho_{0}>0$ are some known constants, and there is a jump in the material parameters
\begin{align*}
&\lambda=\lambda_{0}+\chi_{D_{1}}\psi_{1},\quad D_{1}\subseteq\Omega, \\
&\mu=\mu_{0}+\chi_{D_{2}}\psi_{2},\quad D_{2}\subseteq\Omega, \\
&\rho=\rho_{0}-\chi_{D_{3}}\psi_{3},\quad D_{3}\subseteq\Omega,
\end{align*}
where $\chi_{D_{j}}$ $(j=1,2,3)$ are the characteristic functions of the sets $D_{j}$ and $\psi_{j} |_{D_{j}}\in L_{+}^{\infty}(D_{j}):=\{\psi\in L^{\infty}(D_{j}),\, {\rm ess\,inf}_{D_{j}}\psi>0\}$, so that there is a jump in the material parameters at the boundaries $\partial D_{j}$ of the regions where the material parameters differ from the background values.

The scattering problem we are dealing with is modeled by the following Navier equation:
\begin{equation}\label{a}
\Delta^{*}_{\lambda,\mu} u_{c}+\rho\omega^{2}u_{c}=0, \quad {\rm in} \quad \mathbb{R}^{2}\quad (c:=(\lambda,\mu,\rho)).
\end{equation}
The circular frequency $\omega>0$ and $\Delta^{*}_{\lambda,\mu}$ denotes the Lam\'{e} operator $\mu\Delta+(\mu+\lambda)\nabla(\nabla\cdot)$. Here, $u_{c}=u^{{\rm in}}+u_{c}^{{\rm sc}}$ is the total displacement field, which is a superposition of the given incident plane wave $u^{{\rm in}}$ and the scattered wave $u^{{\rm sc}}_{c}$. By the Helmholtz decomposition theorem, the scattered field $u^{{\rm sc}}_{c}$ can be decomposed as $u^{{\rm sc}}_{c}=u_{p}+u_{s}$, where $u_{p}$ denotes the compressional wave and $u_{s}$ denotes the shear wave, $k_{p}$ is the compressional wave number and $k_{s}$ is the shear wave number. They are given by the following forms respectively:
\begin{equation*}
u_{p}:=u_{p}(\lambda,\mu,\rho) :=-\frac{1}{k_{p}^{2}}\,\nabla(\nabla\cdot u_{c}^{{\rm sc}}),\quad \quad u_{s}:= u_{s}(\lambda,\mu,\rho):=\frac{1}{k_{s}^{2}}\,\overrightarrow{{\rm curl}}\,{\rm curl}\,u_{c}^{{\rm sc}},
\end{equation*}
\begin{equation*}  
k_{p}:=k_{p}(\lambda,\mu,\rho):=\omega\sqrt{\frac{\rho}{2\mu+\lambda}},\quad \quad
k_{s}:=k_{s}(\lambda,\mu,\rho):=\omega\sqrt{\frac{\rho}{\mu}},
\end{equation*}
with
\begin{equation*}
\nabla\cdot u:=\frac{\partial u_{1}}{\partial x_{1}}+\frac{\partial u_{2}}{\partial x_{2}},\quad {\rm curl}\,u=\nabla^{\bot}\cdot u:=\frac{\partial u_{2}}{\partial x_{1}}-\frac{\partial u_{1}}{\partial x_{2}},\quad \overrightarrow{{\rm curl}}:=\left(\frac{\partial}{\partial x_{2}},-\frac{\partial}{\partial x_{1}}\right)^{\top}, \quad u=[u_{1},u_{2}]^{\top}.
\end{equation*}
And $u_{p}$, $u_{s}$ satisfy $\Delta u_{p}+k_{p}^{2}u_{p}=0$ and $\Delta u_{s}+k_{s}^{2}u_{s}= 0$. In addition, the Kupradze radiation condition is required to the scattered field $u_{c}^{{\rm sc}}$, i.e.
\begin{equation} \label{radiation}
\mathop{\lim}\limits_{r\to\infty}\sqrt{r}\Big(\frac{\partial u_{p}}{\partial r}-ik_{p_{0}}u_{p}\Big)=0,~\quad~
\mathop{\lim}\limits_{r\to\infty}\sqrt{r}\Big(\frac{\partial u_{s}}{\partial r}-ik_{s_{0}}u_{s}\Big)=0, \quad r=|x|.
\end{equation}
The radiation condition \eqref{radiation} is assumed to hold in all directions $\hat{x}=x/|x|\in \mathbb{S}:=\{x\in \mathbb{R}^{2},\, |x|=1\}$ and $k_{t_{0}}:= k_{t}(\lambda_{0},\mu_{0},\rho_{0})$ ($t=p,s$). Throughout this paper, the solution to the Navier equation \eqref{a} satisfying the Kupradze radiation condition \eqref{radiation} is called the radiating solution. It is well known that the radiating solution to the Navier equation has the following asymptotic expansions:
\begin{equation*} 
u^{{\rm sc}}_{c}(x)=\frac{e^{ik_{p_{0}}r}}{\sqrt{r}}u_{p}^{\infty}(\hat{x})\hat{x}
+\frac{e^{ik_{s_{0}}r}}{\sqrt{r}}u_{s}^{\infty}(\hat{x})\hat{x}^{\bot}
+\mathcal{O}\Big(\frac{1}{r^{3/2}}\Big),\quad {\rm as} ~r\rightarrow\infty,
\end{equation*}
and
\begin{equation*} 
T_{\lambda,\mu}u^{{\rm sc}}_{c}(x)=\frac{i\omega^{2}}{k_{p_{0}}} \frac{e^{ik_{p_{0}}r}}{\sqrt{r}}{u}_{p}^{\infty}(\hat{x}) \hat{x} +\frac{i\omega^2}{k_{s_{0}}}\frac{e^{ik_{s_{0}}r}}{\sqrt{r}}{u}_{s}^{\infty}(\hat{x}) \hat{x}^{\bot}+\mathcal{O}\Big(\frac{1}{r}\Big),\quad {\rm as} ~r\rightarrow\infty.
\end{equation*}
The stress vector $T_{\lambda,\mu}u$ is defined by
\begin{equation*}
T_{\lambda,\mu}u:=2\mu\frac{\partial u}{\partial\nu}+\lambda\nu~\nabla\cdot u-\mu \nu^{\bot} \nabla^{\bot}\cdot u,
\end{equation*}
where $\nu$ denotes the unit exterior normal vector and $\nu^{\bot}$ is obtained by rotating $\nu$ anticlockwise by $\pi/2$. The functions $u_{p}^{\infty}$, $u_{s}^{\infty}$ are known as compressional and shear far-field patterns of $u^{{\rm sc}}_{c}$, respectively. We will denote the pair of far-field patterns $(u_{p}^{\infty}(\hat{x}),u_{s}^{\infty}(\hat{x}))$ of the corresponding scattered field by $u^{\infty}_{c}(\hat{x})$, i.e.
\begin{equation*}
u^{\infty}_{c}(\hat{x})=(u_{p}^{\infty}(\hat{x}),u_{s}^{\infty}(\hat{x})).
\end{equation*}

Next we introduce the elastic Herglotz wave function with density $g=(g_{p},g_{s})\in [L^{2}(\mathbb{S})]^{2}$ defined by
\begin{equation} \label{Vg(x)}
v_{g}(x)=e^{-i\pi/4}\int_{\mathbb{S}}\left\{\sqrt{\frac{k_{p_{0}}}{\omega}}e^{ik_{p_{0}}d\cdot x}d g_{p}(d)+\sqrt{\frac{k_{s_{0}}}{\omega}}e^{ik_{s_{0}}d\cdot x}d^{\bot}g_{s}(d)\right\}{\rm d}s(d).
\end{equation}
The Hilbert space $[L^{2}(\mathbb{S})]^{2}$ in this paper is equipped with the inner product:
\begin{equation*}
\langle g,h\rangle
=\frac{\omega}{k_{p}}\int_{\mathbb{S}}g_{p}\overline{h_{p}}{\rm d}s+
\frac{\omega}{k_{s}}\int_{\mathbb{S}}g_{s}\overline{h_{s}}{\rm d}s,\quad~g,h\in [L^{2}(\mathbb{S})]^{2}.
\end{equation*}

For the special case of a plane wave incident field $u^{{\rm in}}(x,d)=de^{ik_{p_{0}}x\cdot d} +d^{\bot}e^{ik_{s_{0}}x\cdot d}$, we explicitly indicate the dependence on the incident direction $d\in \mathbb{S}$ by a second argument and accordingly we write $u^{\rm sc}_{c}(x,d)$, $u_{c}(x,d)$ and $u^{\infty}_{c}(\hat{x},d)$ for the corresponding scattered field, total field, and far field pattern of the problem \eqref{a}-\eqref{radiation}, respectively. Define the elastic far-field operator $F_{c}g:[L^{2}(\mathbb{S})]^{2} \longrightarrow [L^{2}(\mathbb{S})]^{2}$ $(c:=(\lambda,\mu,\rho))$ by
\begin{equation} \label{F}
(F_{c}g)(\hat{x})=e^{-i\pi/4}\int_{\mathbb{S}} \left\{\sqrt{\frac{k_{p_{0}}}{\omega}}u^{\infty}_{c}(\hat{x},d)dg_{p}(d) +\sqrt{\frac{k_{s_{0}}}{\omega}}u^{\infty}_{c}(\hat{x},d)d^{\bot}g_{s}(d)\right\}{\rm d}s(d),
\end{equation}

By linearity, for any given function $g\in [L^{2}(\mathbb{S})]^{2}$, the solution to the direct scattering problem \eqref{a}-\eqref{radiation} with incident field of the elastic Herglotz wave function $v_{g}$ defined by \eqref{Vg(x)} is
\begin{equation} \label{ug}
u_{c,g}(x)=\int_{\mathbb{S}}u_{c}(x,d)g(d){\rm d}s(d),\quad x\in\mathbb{R}^{2},
\end{equation}
and the corresponding scattered field
\begin{equation*}
u_{c,g}^{{\rm sc}}(x)=\int_{\mathbb{S}}u^{{\rm sc}}_{c}(x,d)g(d){\rm d}s(d), \quad x\in\mathbb{R}^{2},
\end{equation*}
has the far field pattern $u_{c,g}^{\infty}$ satisfying $u_{c,g}^{\infty}=F_{c}g$.

Finally, we introduce the fundamental solution of the Navier equation \eqref{a} in $\mathbb{R}^2$ space which is given by
\begin{equation*}
\Gamma_{c}(x,y)=\frac{i}{4\mu}H^{(1)}_{0}(k_{s}|x-y|)I
+\frac{i}{4\omega^{2}} \nabla^{\top}_{x}\nabla_{x} \big(H^{(1)}_{0}(k_{s}|x-y|) -H^{(1)}_{0}(k_{p}|x-y|)\big)
\end{equation*}
for $x,y\in \mathbb{R}^{2}$ and $x\neq y$, where $H^{(1)}_{0}(\cdot)$ is the Hankel function of the first kind of order zero and $I$ is the identity matrix. In addition, the subscript $x$ is used to denote differentiation with respect to the corresponding variable. The far field patterns $\Gamma_{p}^{\infty}$ and $\Gamma_{s}^{\infty}$ are given by
\begin{equation*}
\Gamma_{p}^{\infty}(\widehat{x},y)=\frac{1}{\lambda+2\mu} \frac{e^{\frac{i\pi}{4}}}{\sqrt{8\pi k_{p}}}e^{-ik_{p}\widehat{x}\cdot y}J(\widehat{x}), \quad \quad
\Gamma_{s}^{\infty}(\widehat{x},y)=\frac{1}{\mu} \frac{e^{\frac{i\pi}{4}}}{\sqrt{8\pi k_{s}}}e^{-ik_{s}\widehat{x}\cdot y}J(\widehat{x}^{\bot}),
\end{equation*}
where $J(z)=\frac{zz^{\top}}{|z|^{2}}$ for any $z\in \mathbb{R}^{2}$, $z\neq 0$.

\section{A monotonicity relation for the far field operator} \label{sec3}

In this section we derive some monotonicity relations between the parameters $(\lambda,\mu, \rho)$ and the far field operator $F$ that are of fundamental importance in justifying monotonicity based shape reconstruction, and will be needed in the later sections.

\begin{lemma}
Let $\lambda,\mu,\rho\in L^{\infty}_{+}(\mathbb{R}^2)$, and let $B_R(O)$ be a ball containing $\Omega$. Then
\begin{equation} \label{eq0}
\langle g,F_{c}g\rangle=\frac{1}{\sqrt{8\pi\omega}}\int_{\partial B_{R}(O)}\big(T_{0}v_{g}\,\, \overline{u^{\rm sc,+}_{c,g}} -T_{0}\overline{u^{\rm sc,+}_{c,g}}\,\, v_{g}\big){\rm d}s.
\end{equation}
If $\lambda_{j}$, $\mu_{j}$, $\rho_{j}\in L_{+}^{\infty}(\mathbb{R}^2)$, then for any $j,l\in\{1,2\}$ we have
\begin{equation} \label{F0}
\int_{\partial B_R(O)}(u_{c_j, g}^{\rm sc,+}\,\, T_{0}\overline{u_{c_l,g}^{\rm sc,+}}-\overline{u_{c_l,g}^{\rm sc,+}}\,\, T_{0}u_{c_j,g}^{\rm sc,+}){\rm d}s=-2i\omega\langle F_{c_{j}}g,F_{c_{l}}g\rangle.
\end{equation}
\end{lemma}
\begin{proof}
The general form of $E_{\lambda,\mu}(u,v)$ is given by
\begin{align*}
E_{\lambda,\mu}(u,v)=&(2\mu+\lambda)\Big(\frac{\partial u_1}{\partial x_1}\frac{\partial v_1} {\partial x_1}+\frac{\partial u_2}{\partial x_2}\frac{\partial v_2}{\partial x_2}\Big)+\mu\Big(\frac{\partial u_1}{\partial x_2}\frac{\partial v_1} {\partial x_2}+\frac{\partial u_2}{\partial x_1}\frac{\partial v_2}{\partial x_1}\Big)  \\
&+\lambda\Big(\frac{\partial u_1}{\partial x_1}\frac{\partial v_2} {\partial x_2}+\frac{\partial u_2}{\partial x_2}\frac{\partial v_1}{\partial x_1}\Big)
+\mu\Big(\frac{\partial u_1}{\partial x_2}\frac{\partial v_2} {\partial x_1}+\frac{\partial u_2}{\partial x_1}\frac{\partial v_1}{\partial x_2}\Big).
\end{align*}
By Betti Representation Theorem, we have
\begin{equation*}
u^{\rm sc,+}_{c,g}(x)=\int_{\partial B_{R}(O)}\Big[\big[T_{0}\Gamma_{0}(x,y)\big]^{\top} u^{\rm sc,+}_{c,g}(y)-\Gamma_{0}(x,y)T_{0}u^{\rm sc,+}_{c,g}(y)\Big]{\rm d}s(y), \quad x\in \mathbb{R}^{2}\backslash\overline{\Omega},
\end{equation*}
with
\begin{equation*}
T_{j}u:=T_{\lambda_{j},\mu_{j}}u=2\mu_{j}\frac{\partial u}{\partial\nu} +\lambda_{j}\nu~\nabla\cdot u-\mu_{j} \nu^{\bot} \nabla^{\bot}\cdot u, \quad \Delta^{*}_{j}:=\Delta^{*}_{\lambda_{j},\mu_{j}} ,\quad E_{j}:=E_{\lambda_{j},\mu_{j}},
\end{equation*}
\begin{equation*}
\Gamma_{j}(x,y):=\Gamma_{c_{j}}(x,y)=\frac{i}{4\mu_{j}}H^{(1)}_{0}(k_{s_{j}}|x-y|)I
+\frac{i}{4\omega^{2}} \nabla^{\top}_{x}\nabla_{x} \big(H^{(1)}_{0}(k_{s_{j}}|x-y|) -H^{(1)}_{0}(k_{p_{j}}|x-y|)\big),
\end{equation*}
\begin{equation*}
c_{j}:=(\lambda_{j},\mu_{j},\rho_{j}),\quad k_{p_{j}}:=k_{p}(\lambda_{j},\mu_{j},\rho_{j}):=\omega\sqrt{\frac{\rho_{j}}{2\mu_{j} +\lambda_{j}}},\quad k_{s_{j}}:=k_{s}(\lambda_{j},\mu_{j},\rho_{j}):=\omega\sqrt{\frac{\rho_{j}}{\mu_{j}}}.
\end{equation*}
From the asymptotic behavior of the Hankel function $H^{(1)}_{0}(\cdot)$ and the far field patterns $\Gamma_{p}^{\infty}$, $\Gamma_{s}^{\infty}$, it follows that the far field patterns of $u^{\rm sc,+}_{c,g}(x)$ are given by
\begin{equation*}
u_{p}^{\infty}(d)d=\frac{e^{i\pi/4}}{\sqrt{8\pi k_{p_{0}}}}\frac{k_{p_{0}}^{2}}{\omega^{2}} \int_{\partial B_{R}(O)}\Big\{\big[J(d)T_{0}e^{-ik_{p_{0}}d\cdot y}\big]^{\top}u^{\rm sc,+}_{c,g}(y)-J(d)e^{-ik_{p_{0}}d\cdot y}T_{0}u^{\rm sc,+}_{c,g}(y)\Big\}{\rm d}s,
\end{equation*}
\begin{equation*}
u_{s}^{\infty}(d)d^{\bot}=\frac{e^{i\pi/4}}{\sqrt{8\pi k_{s_{0}}}} \frac{k_{s_{0}}^{2}}{\omega^{2}} \int_{\partial B_{R}(O)}\Big\{\big[J(d)^{\bot}T_{0} e^{-ik_{s_{0}}d \cdot y}\big]^{\top}u^{\rm sc,+}_{c,g}(y)-J(d)^{\bot}e^{-ik_{s_{0}}d\cdot y}T_{0}u^{\rm sc,+}_{c,g}(y)\Big\}{\rm d}s
\end{equation*}
where $J(d)^{\bot}=I-J(d)$. Thus,
\begin{align*}
\langle g,F_{c}g\rangle =\frac{\omega}{k_{p_{0}}}\int_{\mathbb{S}}g_{p}\overline{u^{\infty}_{p}}{\rm d}s(d)+\frac{\omega}{k_{s_{0}}}\int_{\mathbb{S}}g_{s}\overline{u^{\infty}_{s}}{\rm d}s
=\frac{1}{\sqrt{8\pi\omega}}\int_{\partial B_{R}(O)}\big(T_{0}v_{g}\,\, \overline{u^{\rm sc,+}_{c,g}} -T_{0}\overline{u^{\rm sc,+}_{c,g}}\,\, v_{g}\big){\rm d}s.
\end{align*}

Let $r>R$, then $u_{c_j,g}^{\rm sc}\in [H_{{\rm loc}}^1(\mathbb{R}^2)]^{2}$ solve (for $c_j:=(\lambda_{j},\mu_{j},\rho_{j})$)
\begin{equation*}
\mu_{0}\Delta u_{c,g}^{\rm sc}+(\mu_{0}+\lambda_{0})\nabla(\nabla\cdot) u_{c,g}^{\rm sc}+\rho_{0}\omega^2 u_{c,g}^{\rm sc}=0 \quad {\rm in}\,B_r(O) \backslash \overline{B_R(O)},
\end{equation*}
and applying Betti's formula we obtain that
\begin{equation} \label{eq1}
\int_{\partial B_r(O)}(u_{c_{j}, g}^{\rm sc,+}\,\, T_{0}\overline{u_{c_{l},g}^{\rm sc,+}}-\overline{u_{c_{l},g}^{\rm sc,+}}\, T_{0}u_{c_{j},g}^{\rm sc,+}){\rm d}s=\int_{\partial B_R(O)}(u_{c_{j},g}^{\rm sc,+}\, T_{0}\overline{u_{c_{l},g}^{\rm sc,+}} -\overline{u_{c_{l},g}^{\rm sc,+}} \, \, T_{0}u_{c_{j},g}^{\rm sc,+}){\rm d}s
\end{equation}
Using the radiation condition \eqref{radiation} and the far field expansion we find that, as $r\rightarrow\infty$,
\begin{equation} \label{eq2}
\int_{\partial B_r(O)}(u_{c_j, g}^{\rm sc,+}\,\, T_{0}\overline{u_{c_l,g}^{\rm sc,+}} -\overline{u_{c_l,g}^{\rm sc,+}}\,\, T_{0}u_{c_j,g}^{\rm sc,+}){\rm d}s=-2i\omega\langle F_{c_{j}}g,F_{c_{l}}g\rangle.
\end{equation}
Substituting \eqref{eq2} into \eqref{eq1} and letting $r\rightarrow \infty$ finally gives \eqref{F0}.
\end{proof}

The next tool we will use to prove the monotonicity relation for the far field operator is the following integral identity.

\begin{lemma}
If $\lambda_{j}$, $\mu_{j}$, $\rho_{j}\in L_{+}^{\infty}(\mathbb{R}^2)$ and $B_R(O)$ is a ball containing $\Omega$. Then, for any $g\in [L^2(\mathbb{S})]^{2}$, it holds that
\begin{align} \label{maineq0}
&\sqrt{8\pi\omega}\left(\langle F_{c_{1}}g,g\rangle-\langle g,F_{c_{2}}g\rangle\right) -2i\omega\langle F_{c_{1}}g,F_{c_{2}}g\rangle \\
=&\int_{\partial B_R(O)}(\overline{u}_{c_2,g}-\overline{u}_{c_1,g})(T_{2}u_{c_2,g} -T_{1}u_{c_1,g}) {\rm d}s \nonumber\\ &+\int_{B_R(O)}\rho_{2}\omega^{2}|u_{c_2,g}-u_{c_1,g}|^{2}
-E_{2}(u_{c_1,g}-u_{c_2,g},\overline{u}_{c_1,g}-\overline{u}_{c_2,g}){\rm d}y \nonumber\\ &+\int_{B_R(O)}E_{2}(\overline{u}_{c_1,g},u_{c_1,g}) -E_{1}(\overline{u}_{c_1,g},u_{c_1,g}) +(\rho_{1}-\rho_{2})\omega^{2}|u_{c_1,g}|^{2}{\rm d}y. \nonumber
\end{align}
\end{lemma}
\begin{proof}
The identity \eqref{eq0} and \eqref{F0} (with $j=1$ and $l=2$ ) immediately imply that
\begin{align*}
&\sqrt{8\pi\omega}\left(\langle F_{c_{1}}g,g\rangle-\langle g,F_{c_{2}}g\rangle\right) -2i\omega\langle F_{c_{1}}g,F_{c_{2}}g\rangle \\
=&\int_{\partial B_{R}(O)}\big(T_{0}\overline{v_{g}}\,u^{\rm sc,+}_{c_{1},g} -T_{0}u^{\rm sc,+}_{c_{1},g}\,\overline{v_{g}}\big){\rm d}s-\int_{\partial B_{R}(O)}\big(T_{0}v_{g}\, \overline{u^{\rm sc,+}_{c_{2},g}}-T_{0}\overline{u^{\rm sc,+}_{c_{2},g}}\,v_{g}\big){\rm d}s\\
&+\int_{\partial B_R(O)}(u_{c_1,g}^{\rm sc,+}\,\, T_{0}\overline{u_{c_2,g}^{\rm sc,+}}-\overline{u_{c_2,g}^{\rm sc,+}}\,\, T_{0}u_{c_1,g}^{\rm sc,+}){\rm d}s \\
=&\int_{\partial B_R(O)}(u_{c_1,g}^{\rm sc,+}\,\, T_{2}\overline{u_{c_2,g}^{-}} -\overline{u_{c_2,g}^{-}}\,\, T_{0}u_{c_1,g}^{\rm sc,+}){\rm d}s-\int_{\partial B_{R}(O)}\big(T_{0}v_{g}\, \overline{u^{\rm sc,+}_{c_{2},g}}-T_{0}\overline{u^{\rm sc,+}_{c_{2},g}}\, v_{g}\big){\rm d}s \\
=&\int_{\partial B_R(O)}(u_{c_1,g}^{-}\,\, T_{2}\overline{u_{c_2,g}^{-}} -\overline{u_{c_2,g}^{-}}\,\, T_{1}u_{c_1,g}^{-}){\rm d}s-\int_{\partial B_R(O)}(v_{g}\,\, T_{2}\overline{u_{c_2,g}^{-}} -\overline{u_{c_2,g}^{-}}\,\, T_{0}v_{g}){\rm d}s \\
&-\int_{\partial B_{R}(O)}\big(T_{0}v_{g}\, \overline{u^{\rm -}_{c_{2},g}} -T_{2}\overline{u^{-}_{c_{2},g}}\, v_{g}\big){\rm d}s+\int_{\partial B_{R}(O)}\big(T_{0}v_{g}\, \overline{v_{g}}-T_{0}\overline{v_{g}}\, v_{g}\big){\rm d}s \\
=&\int_{\partial B_R(O)}(u_{c_1,g}^{-}\,\, T_{2}\overline{u_{c_2,g}^{-}} -\overline{u_{c_2,g}^{-}}\,\, T_{1}u_{c_1,g}^{-}){\rm d}s,
\end{align*}
where we have used the transmission boundary conditions
\begin{equation*}
u_{c_j,g}^{\rm sc,+}+v_{g}=u_{c_j,g}^{-}, \quad T_{0}u^{\rm sc,+}_{c_{j},g}+T_{0}v_{g}= T_{j}u^{-}_{c_{j},g}  \quad {\rm on} \quad \partial B_R(O).
\end{equation*}
For notational simplicity, we omit the superscript, that is,
\begin{align*}
&\int_{\partial B_R(O)}(u_{c_1,g}^{-}\,\, T_{2}\overline{u_{c_2,g}^{-}} -\overline{u_{c_2,g}^{-}}\,\, T_{1}u_{c_1,g}^{-}){\rm d}s:=\int_{\partial B_R(O)}(u_{c_1,g} \,\, T_{2}\overline{u}_{c_2,g} -\overline{u}_{c_2,g} T_{1}u_{c_1,g}){\rm d}s \\
=&\int_{\partial B_R(O)}(\overline{u}_{c_2,g}-\overline{u}_{c_1,g})(T_{2}u_{c_2,g} -T_{1}u_{c_1,g}){\rm d}s \\ &+\int_{\partial B_R(O)}(u_{c_1,g}T_{2}\overline{u}_{c_2,g} -\overline{u}_{c_2,g}T_{2}u_{c_2,g} +\overline{u}_{c_1,g}T_{2}u_{c_2,g} -\overline{u}_{c_1,g}\,T_{1}u_{c_1,g}){\rm d}s.
\end{align*}
Applying Betti's formula yields
\begin{align*}
&\int_{\partial B_R(O)}(u_{c_1,g}T_{2}\overline{u}_{c_2,g} -\overline{u}_{c_2,g}T_{2}u_{c_2,g} +\overline{u}_{c_1,g}T_{2}u_{c_2,g} -\overline{u}_{c_1,g}\,T_{1}u_{c_1,g}){\rm d}s \\
=&\int_{B_R(O)}E_{2}(u_{c_1,g}-u_{c_2,g},\overline{u}_{c_2,g}) +E_{2}(\overline{u}_{c_1,g},u_{c_2,g}-u_{c_1,g})+E_{2}(\overline{u}_{c_1,g}, u_{c_1,g}) -E_{1}(\overline{u}_{c_1,g},u_{c_1,g}){\rm d}y \\
&+\omega^{2}\int_{B_R(O)}\rho_{2}\,u_{c_2,g}\overline{u}_{c_2,g} -\rho_{2}\, u_{c_1,g}\overline{u}_{c_2,g} -\rho_{2}\,\overline{u}_{c_1,g}u_{c_2,g} +\rho_{1}\,\overline{u}_{c_1,g}u_{c_1,g}{\rm d}y \\
=&\int_{B_R(O)}\rho_{2}\omega^{2}|u_{c_2,g}-u_{c_1,g}|^{2} -E_{2}(u_{c_1,g}-u_{c_2,g},\overline{u}_{c_1,g}-\overline{u}_{c_2,g}){\rm d}y \\
&+\int_{B_R(O)}E_{2}(\overline{u}_{c_1,g},u_{c_1,g}) -E_{1}(\overline{u}_{c_1,g},u_{c_1,g}) +(\rho_{1}-\rho_{2})\omega^{2}|u_{c_1,g}|^{2}{\rm d}y,
\end{align*}
Consequently, we obtain that
\begin{align*}
&\sqrt{8\pi\omega}\left(\langle F_{c_{1}}g,g\rangle-\langle g,F_{c_{2}}g\rangle\right) -2i\omega\langle F_{c_{1}}g,F_{c_{2}}g\rangle \\
=&\int_{\partial B_R(O)}(\overline{u}_{c_2,g}-\overline{u}_{c_1,g})(T_{2}u_{c_2,g} -T_{1}u_{c_1,g}) {\rm d}s \\
&+\int_{B_R(O)}\rho_{2}\omega^{2}|u_{c_2,g}-u_{c_1,g}|^{2}
-E_{2}(u_{c_1,g}-u_{c_2,g},\overline{u}_{c_1,g}-\overline{u}_{c_2,g}){\rm d}y \\ &+\int_{B_R(O)}E_{2}(\overline{u}_{c_1,g},u_{c_1,g}) -E_{1}(\overline{u}_{c_1,g},u_{c_1,g}) +(\rho_{1}-\rho_{2})\omega^{2}|u_{c_1,g}|^{2}{\rm d}y .
\end{align*}
\end{proof}

\begin{lemma}[Theorem 2 in \cite{S2005}]
The scattering matrix given by $S_{c}=I+i\sqrt{\frac{\omega}{2\pi}}F_{c}$ is a unitary operator, i.e. $S_{c}^{*}S_{c}=S_{c}S^{*}_{c}=I$.
\end{lemma}

\begin{remark}
Since the adjoint of the scattering operator $S_{c_{1}}$ is given by
\begin{equation*}
S_{c_{1}}^{*}=I-i\sqrt{\frac{\omega}{2\pi}}F_{c_{1}}^{*},
\end{equation*}
we find that
\begin{equation*}
S_{c_{1}}^{*}\left(F_{c_{2}}-F_{c_{1}}\right)=F_{c_{2}}-F_{c_{1}}-i\sqrt{\frac{\omega}{2\pi}} \left(F_{c_{1}}^{*}F_{c_{2}}-F_{c_{1}}^{*}F_{c_{1}}\right),
\end{equation*}
and accordingly
\begin{equation*}
\Re\left(S_{c_{1}}^{*}\left(F_{c_{2}}-F_{c_{1}}\right)\right)=\Re\left( F_{c_{2}}-F_{c_{1}} -i\sqrt{\frac{\omega}{2\pi}}F_{c_{1}}^{*}F_{c_{2}} \right).
\end{equation*}
Therefore the real part of the first two terms on the left-hand side of \eqref{maineq0} fulfills
\begin{align*}
&\Re\left(\sqrt{8\pi\omega}\left(\left\langle F_{c_{1}}g,g\right\rangle-\left\langle g,F_{c_{2}}g\right\rangle\right) -2i\omega\left\langle F_{c_{1}}g,F_{c_{2}}g\right\rangle \right) \\
=&-\sqrt{8\pi\omega}\,\,\Re\left(\left\langle g,F_{c_{2}}g\right\rangle-\left\langle F_{c_{1}}g,g\right\rangle +i\sqrt{\frac{\omega}{2\pi}}\left\langle F_{c_{1}}g,F_{c_{2}}g \right\rangle \right) \\
=&-\sqrt{8\pi\omega}\,\,\Re\left(\left\langle F_{c_{2}}g,g\right\rangle-\left\langle F_{c_{1}}g,g\right\rangle -i\sqrt{\frac{\omega}{2\pi}}\left\langle F_{c_{2}}g,F_{c_{1}}g \right\rangle \right) \\
=&-\sqrt{8\pi\omega}\,\,\Re\left\langle S_{c_{1}}^{*}\left(F_{c_{2}}-F_{c_{1}}\right)g, g\right\rangle.
\end{align*}
That is,
\begin{align} \label{maineq1}
&\sqrt{8\pi\omega}\,\,\Re\left\langle S_{c_{1}}^{*}\left(F_{c_{2}}-F_{c_{1}}\right)g, g\right\rangle+\int_{B_R(O)}E_{\lambda_{2}-\lambda_{1},\mu_{2}-\mu_{1}}(\overline{u}_{c_1,g}, u_{c_1,g})+(\rho_{1}-\rho_{2})\omega^{2}|u_{c_1,g}|^{2}{\rm d}y \\
=&\int_{B_R(O)}E_{2}(u_{c_1,g}-u_{c_2,g},\overline{u}_{c_1,g}-\overline{u}_{c_2,g})-\rho_{2} \omega^{2}|u_{c_2,g}-u_{c_1,g}|^{2}{\rm d}y \nonumber \\
&-\Re\left(\int_{\partial B_R(O)}(\overline{u}_{c_2,g} -\overline{u}_{c_1,g}) (T_{2}u_{c_2,g}-T_{1}u_{c_1,g}){\rm d}s\right).\nonumber
\end{align}
\end{remark}

Next we consider the right-hand side of \eqref{maineq1}, and we show that it is nonnegative if $g$ belongs to the complement of a certain finite dimensional subspace $V\subseteq [L^{2}(\mathbb{S})]^{2}$. To that end we denote by $J:[H^{1}(B_{R}(O))]^{2} \rightarrow [L^{2}(B_{R}(O))]^{2}$ the compact embedding for any ball $B_{R}(O)$ containing $\Omega$, and accordingly we define, for any $\rho\in L_{+}^{\infty}(\mathbb{R}^{2})$, the operator $K:[H^{1}(B_{R}(O))]^{2}\rightarrow [H^{1}(B_{R}(O))]^{2}$ by
\begin{equation*}
Kv:=J^{*}Jv,
\end{equation*}
and $K_{\rho}:[H^{1}(B_{R}(O))]^{2}\rightarrow [H^{1}(B_{R}(O))]^{2}$ by
\begin{equation*}
K_{\rho}v:=\rho J^{*}Jv.
\end{equation*}
The special identity operator $I_{\lambda,\mu}:[H^{1}(B_{R}(O))]^{2} \rightarrow [H^{1}(B_{R}(O))]^{2}$ is defined by
\begin{equation*}
\langle I_{\lambda,\mu}v,w\rangle_{[H^{1}(B_{R}(O))]^{2}}= \int_{B_R(O)}E_{\lambda,\mu}(v,\overline{w})+v\,\overline{w}\,{\rm d}y.
\end{equation*}
Then $K$ and $K_{\rho}$ are compact self-adjoint linear operators, and, for any $v\in [H^{1}(B_{R}(O))]^{2}$,
\begin{equation*}
\langle(I_{\lambda,\mu}-K-\omega^{2}K_{\rho})v,v\rangle_{[H^{1}(B_{R}(O))]^{2}}= \int_{B_R(O)}E_{\lambda,\mu}(v,\overline{v})-\rho\omega^{2}|v|^{2}{\rm d}y.
\end{equation*}

For $0<\varepsilon<R$ we denote by $N_{\varepsilon}: [H^1(B_R(O))]^{2}\rightarrow [L^2(\partial B_R(O))]^{2}$ the bounded linear operator that maps $v\in [H^1(B_R(O))]^{2}$ to the stress vector $T_{0}v_{\varepsilon}$ on $\partial B_R(O)$ of the radiating solution to the exterior boundary value problem
\begin{equation*}
\Delta^{*}_{0} v_{\varepsilon}+\rho_{0}\omega^{2}v_{\varepsilon}=0 \quad \text { in } \mathbb{R}^2\backslash \overline{B_{R-\varepsilon}(O)}, \quad v_{\varepsilon}=v \quad \text { on } \partial B_{R-\varepsilon}(O),
\end{equation*}
and $\Lambda: [L^2(\partial B_R(O))]^{2} \rightarrow [L^2(\partial B_R(O))]^{2}$ denotes the compact exterior Neumann-to-Dirichlet operator that maps $\psi\in [L^2(\partial B_R(O))]^{2}$ to the trace $w|_{\partial B_{R}(O)}$ of the radiating solution to
\begin{equation*}
\Delta^{*}_{0}w+\rho_{0}\omega^{2}w=0 \quad \text { in } \mathbb{R}^2\backslash \overline{B_R(O)}, \quad T_{0}w=\psi \quad \text { on } \partial B_R(O).
\end{equation*}
Then,
\begin{equation*}
N_{\varepsilon}v=T_{0}v|_{\partial B_R(O)} \quad \text { and } \quad \Lambda N_{\varepsilon}v=v|_{\partial B_R(O)},
\end{equation*}
and accordingly
\begin{equation*}
\langle N_{\varepsilon}^* \Lambda N_{\varepsilon}v,v\rangle_{[H^1(B_R(O))]^{2}} =\langle\Lambda N_{\varepsilon} v, N_{\varepsilon} v\rangle_{[L^2(\partial B_R(O))]^{2}}=\int_{\partial B_R(O)} v\,\, T_{0}\overline{v}\,\mathrm{~d} s
\end{equation*}
for any $v\in [H^1(B_R(O))]^{2}$ that can be extended to a radiating solution of the Navier equation
\begin{equation*}
\Delta^{*}_{0}v+\rho_{0}\omega^{2}v=0 \quad \text { in } \mathbb{R}^2 \backslash \overline{B_{R-\varepsilon}(O)}.
\end{equation*}

\begin{lemma}  \label{lem0}
Let $\lambda_{j}$, $\mu_{j}$, $\rho_{j}\in L_{+}^{\infty}(\mathbb{R}^2)$ and let $B_R(O)$ be a ball containing $\Omega$. Then there exists a finite dimensional subspace $V \subseteq [L^2(\mathbb{S})]^{2}$ such that
\begin{align*}
&\int_{B_R(O)}E_{2}(u_{c_{1},g}-u_{c_{2},g},\overline{u}_{c_{1},g}-\overline{u}_{c_{2},g}) -\rho_{2}\omega^{2}|u_{c_{2},g}-u_{c_{1},g}|^{2}{\rm d}y \\
&-\Re\left(\int_{\partial B_R(O)}(\overline{u}_{c_{2},g}-\overline{u}_{c_{1},g}) (T_{2}u_{c_{2},g}-T_{1}u_{c_{1},g}){\rm d}s\right)\geq0, \quad \text{ for all } g\in V^{\bot}.
\end{align*}
\end{lemma}
\begin{proof}
Let $\varepsilon>0$ be sufficiently small, so that $\Omega \subseteq B_{R-\varepsilon}(O)$. Then
\begin{align*}
&\int_{B_R(O)}E_{2}(u_{c_{1},g}-u_{c_{2},g},\overline{u}_{c_{1},g}-\overline{u}_{c_{2},g}) -\rho_{2}\omega^{2}|u_{c_{2},g}-u_{c_{1},g}|^{2}{\rm d}y \\
&-\Re\left(\int_{\partial B_R(O)}(\overline{u}_{c_{2},g}-\overline{u}_{c_{1},g}) (T_{2}u_{c_{2},g}-T_{1}u_{c_{1},g}){\rm d}s\right) \\
=&\int_{B_R(O)}E_{2}(w,\overline{w})-\rho_{2} \omega^{2}|w|^{2}{\rm d}y-\Re\left(\int_{\partial B_R(O)}\overline{w}\,\,T_{0}w\,{\rm d}s\right) \\
=&\langle(I_{\lambda_{2},\mu_{2}}-K-\omega^2 K_{\rho_{2}}-{\Re}(N_{\varepsilon}^*\Lambda N_{\varepsilon}))w,w\rangle_{[H^1(B_R(O))]^{2}}
\end{align*}
where $w|_{B_{R}(O)}:=u_{c_{2},g}^{{\rm sc,-}}-u_{c_{1},g}^{{\rm sc,-}}$ and $w|_{\partial B_{R}(O)}:=u_{c_{2},g}^{{\rm sc,+}}-u_{c_{1},g}^{{\rm sc,+}}$.

Let $W$ be the sum of eigenspaces of the compact self-adjoint operator $K+\omega^2 K_{\rho_2} +{\rm Re}(N_{\varepsilon}^* \Lambda N_{\varepsilon})$ associated to eigenvalues larger than 1. Then $W$ is finite dimensional and
\begin{equation*}
\langle(I_{\lambda_{2},\mu_{2}}-K-\omega^2 K_{\rho_{2}}-{\Re}(N_{\varepsilon}^*\Lambda N_{\varepsilon}))w,w\rangle_{[H^1(B_R(O))]^{2}}\geq0 \quad {\rm for~all~}w\in W^{\perp}.
\end{equation*}
For $j=1,2$ we denote by $\mathcal{S}_j: [L^2(\mathbb{S})]^{2}\rightarrow [H^1(B_R(O))]^{2}$ the bounded linear operator that maps $g\in [L^2(\mathbb{S})]^{2}$ to the restriction of the scattered field $u_{c_{j},g}^{{\rm sc,-}}$ on $B_R(O)$. Then $w|_{B_{R}(O)}=(\mathcal{S}_{2} -\mathcal{S}_{1})g$.

Since, for any $g\in [L^2(\mathbb{S})]^{2}$,
\begin{equation*}
(\mathcal{S}_{2}-\mathcal{S}_{1})g\in W^{\bot} \quad\quad \text{ if and only if } \quad \quad  g\in ((\mathcal{S}_{2}-\mathcal{S}_{1})^{*}W)^{\bot},
\end{equation*}
and of course ${\rm dim}((\mathcal{S}_{2}-\mathcal{S}_{1})^{*}W)\leq {\rm dim}(W)<\infty$, choosing $V:=(\mathcal{S}_{2}-\mathcal{S}_{1})^{*}W$ ends the proof.
\end{proof}

Applying the above Lemma \ref{lem0} in the equality \eqref{maineq1} yields the main monotonicity inequalities \eqref{main1}-\eqref{main2} we will be using.
\begin{theorem}  \label{main}
Let $\lambda_{j}$, $\mu_{j}$, $\rho_{j}\in L_{+}^{\infty}(\mathbb{R}^2)$. Then there exists a finite dimensional subspace $V\subseteq [L^2(\mathbb{S})]^{2}$ such that
\begin{equation} \label{main1}
\sqrt{8\pi\omega}\,\,\Re\langle S_{c_{1}}^{*}(F_{c_{2}}-F_{c_{1}})g, g\rangle \geq \int_{\mathbb{R}^{2}}E_{\lambda_{1}-\lambda_{2},\mu_{1}-\mu_{2}}(\overline{u}_{c_{1},g},u_{c_{1},g}) +(\rho_{2}-\rho_{1})\omega^{2}|u_{c_{1},g}|^{2}{\rm d}y,
\end{equation}
for all $g\in V^{\bot}$. In particular,
\begin{equation} \label{main2}
\lambda_{1}\geq\lambda_{2}, \, \mu_{1}\geq\mu_{2}, \, \rho_{2}\geq \rho_{1} \quad {\rm implies} \quad \Re(S_{c_{1}}^{*}F_{c_{2}}) \geq_{{\rm fin}} \Re(S_{c_{1}}^{*}F_{c_{1}}).
\end{equation}
\end{theorem}

\begin{remark} \label{rem1}
Since the scattering operators $S_{c_{1}}$ and $S_{c_{2}}$ are unitary, we find that
\begin{align*}
&S_{c_{1}}^{*}(F_{c_{2}}-F_{c_{1}})=i\sqrt{\frac{2\pi}{\omega}}S_{c_{1}}^{*}(S_{c_{1}}-S_{c_{2}}) =i\sqrt{\frac{2\pi}{\omega}}(I-S_{c_{1}}^{*}S_{c_{2}})\\
&=\left( i\sqrt{\frac{2\pi}{\omega}}(S_{c_{2}}^{*}S_{c_{1}}-I)\right)^{*} =\left( i\sqrt{\frac{2\pi}{\omega}}S_{c_{2}}^{*}(S_{c_{1}}-S_{c_{2}})\right)^{*} =\left( S_{c_{2}}^{*}(F_{c_{2}}-F_{c_{1}})\right)^{*}.
\end{align*}
Recalling that the eigenvalues of a compact linear operator and of its adjoint are complex conjugates of each other, we conclude that the spectra of $\Re(S_{c_{1}}^{*}(F_{c_{2}} -F_{c_{1}}))$ and $\Re(S_{c_{2}}^{*}(F_{c_{2}}-F_{c_{1}}))$ coincide. Consequently, the monotonicity relations \eqref{main1}-\eqref{main2} remain true if we replace $S_{c_{1}}^{*}$ by $S_{c_{2}}^{*}$ in these formulas.
\end{remark}

Note that by interchanging $\lambda_{1}$, $\mu_{1}$, $\rho_{1}$ and $\lambda_{2}$, $\mu_{2}$, $\rho_{2}$, except for $S_{c_{1}}^{*}$ (see Remark \ref{rem1}), we may restate Theorem \ref{main} as follows.

\begin{corollary} \label{cor1}
Let $\lambda_{j}$, $\mu_{j}$, $\rho_{j}\in L_{+}^{\infty}(\mathbb{R}^2)$. Then there exists a finite dimensional subspace $V\subseteq[L^2(\mathbb{S})]^{2}$ such that
\begin{equation} \label{main3}
\sqrt{8\pi\omega}\,\,\Re\langle S_{c_{1}}^{*}(F_{c_{2}}-F_{c_{1}})g, g\rangle \leq \int_{\mathbb{R}^{2}}E_{\lambda_{1}-\lambda_{2},\mu_{1}-\mu_{2}}(\overline{u}_{c_{2},g}, u_{c_{2},g}) +(\rho_{2}-\rho_{1})\omega^{2}|u_{c_{2},g}|^{2}{\rm d}y,
\end{equation}
for all $g\in V^{\bot}$.
\end{corollary}

\section{Localized potentials for the Navier equation}

In this section we establish the existence of localized wave functions that have arbitrarily large norm on some prescribed region $B\subseteq\mathbb{R}^{2}$ while at the same time having arbitrarily small norm in a different region $D\subseteq\mathbb{R}^{2}$, assuming that $\mathbb{R}^{2} \backslash\overline{D}$ is connected. These will be utilized to establish a rigorous characterization of the region $\Omega={\rm supp}(\lambda-\lambda_{0})\cup {\rm supp}(\mu-\mu_{0})\cup {\rm supp}(\rho-\rho_{0})$ where the material parameters differ from background in terms of the far field operator using the monotonicity relations from Theorem \ref{main} and Corollary \ref{cor1} in section 5 below.

\begin{lemma} \label{lem1}
Suppose that $\lambda$, $\mu$, $\rho\in L_{+}^{\infty}(\mathbb{R}^2)$ and assume that $D \subseteq \mathbb{R}^2$ is open and bounded. We define
\begin{equation*}
L_{c,D}: [L^2(\mathbb{S})]^{2} \rightarrow [H^1(D)]^{2}, \quad g \mapsto u_{c,g}|_D,
\end{equation*}
\begin{equation*}
L_{c,D}^{(1)}: [L^2(\mathbb{S})]^{2} \rightarrow [L^2(D)]^{2}, \quad g \mapsto u_{c,g}|_D,
\end{equation*}
\begin{equation*}
L_{c,D}^{(2)}: [L^2(\mathbb{S})]^{2} \rightarrow L^2(D), \quad g \mapsto \nabla\cdot u_{c,g}|_D,
\end{equation*}
\begin{equation*}
L_{c,D}^{(3)}: [L^2(\mathbb{S})]^{2} \rightarrow [L^2(D)]^{2\times2}, \quad g \mapsto \widehat{\nabla} u_{c,g}|_D,
\end{equation*}
where $u_{c,g} \in [H_{{\rm loc}}^1(\mathbb{R}^2)]^{2}$ is given by \eqref{ug} and $\widehat{\nabla}u:=\frac{1}{2} \left( \nabla u+(\nabla u)^{\top}\right)$. Then $L_{c,D}$, $L_{c,D}^{(1)}$, $L_{c,D}^{(2)}$ and $L_{c,D}^{(3)}$ are linear operators and their dual operator are given by
\begin{equation*}
L_{c,D}^{'}: [H^1(D)'\,]^{2} \rightarrow [L^2(\mathbb{S})]^{2}, \quad f_{0} \mapsto S_c^* w_{0}^{\infty};
\end{equation*}
\begin{equation*}
L_{c,D}^{(1)'}: [L^2(D)]^{2} \rightarrow [L^2(\mathbb{S})]^{2}, \quad f_{1} \mapsto S_c^* w_{1}^{\infty};
\end{equation*}
\begin{equation*}
L_{c,D}^{(2)'}: L^2(D) \rightarrow [L^2(\mathbb{S})]^{2}, \quad f_{2} \mapsto S_c^* w_{2}^{\infty};
\end{equation*}
\begin{equation*}
L_{c,D}^{(3)'}: [L^2(D)]^{2\times2} \rightarrow [L^2(\mathbb{S})]^{2}, \quad f_{3} \mapsto S_c^* w_{3}^{\infty}
\end{equation*}
where $S_c$ denotes the scattering operator, and $w_{j}^{\infty}\in [L^2(\mathbb{S})]^{2}$ $(j=0,1,2,3)$ is the far field pattern of the radiating solution $w_{j}\in [H_{{\rm loc}}^1(\mathbb{R}^2)]^{2}$ to
\begin{equation*} 
\sqrt{8\pi\omega}(f_{0},v)=\int_{B_R(O)}\left(E_{\lambda,\mu}(w_{0},\overline{v})-\rho\omega^2 w_{0} \overline{v}\right){\rm d}x-\int_{\partial B_R(O)} \overline{v}\, T_{\lambda,\mu}w_{0}\mathrm{~d}s, 
\end{equation*}
\begin{equation*} 
\sqrt{8\pi\omega}\int_{B_R(O)}f_{1}\,v\,{\rm d}x=\int_{B_R(O)}\left(E_{\lambda,\mu}(w_{1}, \overline{v}) -\rho\omega^2 w_{1} \overline{v}\right){\rm d}x-\int_{\partial B_R(O)} \overline{v}\, T_{\lambda,\mu}w_{1}\mathrm{~d}s, 
\end{equation*}
\begin{equation*} 
\sqrt{8\pi\omega}\int_{B_R(O)}f_{2}\,\nabla\cdot v\,{\rm d}x =\int_{B_R(O)}\left(E_{\lambda,\mu}(w_{2},\overline{v}) -\rho\omega^2 w_{2} \overline{v}\right){\rm d}x -\int_{\partial B_R(O)} \overline{v}\, T_{\lambda,\mu}w_{2}\mathrm{~d}s, 
\end{equation*}
\begin{equation*} 
\sqrt{8\pi\omega}\int_{B_R(O)}f_{3}:\,\widehat{\nabla}v\,{\rm d}x =\int_{B_R(O)}\left(E_{\lambda,\mu}(w_{3},\overline{v}) -\rho\omega^2 w_{3} \overline{v}\right){\rm d}x -\int_{\partial B_R(O)} \overline{v}\, T_{\lambda,\mu}w_{3}\mathrm{~d}s, 
\end{equation*}
for all $v\in[H^1(B_R(O))]^{2}$ with $D\subseteq B_R(O)$ (the round brackets denote the dual pairing between $H^1(D)$ and its dual space $H^1(D)'$, and $\mathrm{A}:\mathrm{B} =\sum\limits_{i,j=1}^{2}a_{ij}b_{ij}$ for matrices $\mathrm{A}=(a_{ij})$ and $\mathrm{B}=(b_{ij})$).
\end{lemma}
\begin{proof}
The representation formula for the total field in \eqref{ug} shows that $L_{c, D}$ is a Fredholm integral operator with square integrable kernel and therefore linear from $[L^2(\mathbb{S})]^{2}$ to $[H^1(D)]^{2}$.

Applying Betti's formula and the representation formula for the far field pattern $w_{0}^{\infty}$ of the radiating solution $w_{0}$, we find that, for any $g\in[L^2(\mathbb{S})]^{2}$ and $f_{0}\in[H^{1}(D)']^{2}$,
\begin{align*}
&\sqrt{8\pi\omega}(L_{c,D}g,f_{0}) =\int_{B_R(O)}\left(E_{\lambda,\mu}(\overline{w_{0}}, u_{c,g}) -\rho\omega^2 \overline{w_{0}}u_{c,g}\right) {\rm d}x-\int_{\partial B_R(O)}u_{c,g}\,\, T_{\lambda,\mu}\overline{w_{0}}\,{\rm d}s \\
=&\int_{\partial B_R(O)}\left(\overline{w_{0}}\,\, T_{\lambda,\mu}u_{c,g}-u_{c,g}\,\, T_{\lambda,\mu}\overline{w_{0}}\right)\,{\rm d}s \\
=&\int_{\partial B_R(O)}\left(\overline{w_{0}}\,\, T_{\lambda,\mu}v_{g}-v_{g}\,\, T_{\lambda,\mu}\overline{w_{0}}\right)\,{\rm d}s
+\int_{\partial B_R(O)}\left(\overline{w_{0}}\,\, T_{\lambda,\mu}u_{c,g}^{\rm sc}-u_{c,g}^{\rm sc}\,\, T_{\lambda,\mu}\overline{w_{0}}\right)\,{\rm d}s \\
=&\sqrt{8\pi\omega}\langle g,w_{0}^{\infty}\rangle+2i\omega\langle F_{c}g,w_{0}^{\infty}\rangle =\sqrt{8\pi\omega}\left\langle \left(I+i\sqrt{\frac{\omega}{2\pi}}F_{c}\right)g, w_{0}^{\infty} \right\rangle \\
=&\sqrt{8\pi\omega}\, \left\langle S_{c}\,g,w_{0}^{\infty} \right\rangle =\sqrt{8\pi\omega}\, \left\langle g,S_{c}^{*}\,w_{0}^{\infty} \right\rangle.
\end{align*}
That is, $L_{c,D}^{'}f_{0}=S_{c}^{*}\,w_{0}^{\infty}$. The calculations for $L_{c, D}^{(1)'}$, $L_{c, D}^{(2)'}$ and $L_{c, D}^{(3)'}$ are the same, we omit it here for brevity. The proof is complete.
\end{proof}

\begin{lemma} \label{lem2}
Suppose that $\lambda$, $\mu$, $\rho\in L_{+}^{\infty}(\mathbb{R}^2)$ and let $B$, $D \subseteq \mathbb{R}^2$ be open and bounded such that $\mathbb{R}^{2}\backslash(\overline{B}\cup \overline{D})$ is connected and $\overline{B}\cap\overline{D}=\emptyset$. Then,
\begin{equation*}
\mathcal{R}(L_{c,B}^{(\ell)'})\cap\mathcal{R}(L_{c,D}^{'})=\{0\} \quad {\rm and} \quad \mathcal{R}(L_{c,B}^{'})\cap\mathcal{R}(L_{c,D}^{'})=\{0\}  \quad  (\ell=1,2,3).
\end{equation*}
\end{lemma}
\begin{proof}

For simplicity, we focus on the case $\ell=1$ since the proof is similar. Suppose that $h \in \mathcal{R}(L_{c,B}^{(1)'}) \cap \mathcal{R}(L_{c,D}^{'})$. Then Lemma \ref{lem1} shows that there exist $f_B \in [L^{2}(B)]^{2}$, $f_D \in [H^{1}(D)']^{2}$, and $w_B$, $w_D \in [H_{{\rm loc}}^1(\mathbb{R}^2)]^{2}$ such that the far field patterns $w_B^{\infty}$ and $w_D^{\infty}$ of the radiating solutions to
\begin{equation*}
\Delta^{*}_{\lambda,\mu}w_B+\rho\omega^{2}w_B=0\quad {\rm in}\,\, \mathbb{R}^2\backslash\overline{B} \quad {\rm and} \quad \Delta^{*}_{\lambda,\mu}w_D+\rho\omega^{2}w_D=0\quad {\rm in}\,\,\mathbb{R}^2\backslash\overline{D}
\end{equation*}
satisfy
\begin{equation*}
w_B^{\infty}=w_D^{\infty}=S_c h.
\end{equation*}
Rellich's lemma and unique continuation guarantee that $w_B=w_D$ in $\mathbb{R}^2 \backslash(\overline{B} \cup \overline{D})$. Hence we may define $w\in [H_{{\rm loc }}^1(\mathbb{R}^2)]^{2}$ by
\begin{equation*}
w:=\begin{cases}w_B=w_D & \text { in } \mathbb{R}^2\backslash(\overline{B} \cup \overline{D}), \\ w_B & \text { in } D, \\ w_D & \text { in } B,\end{cases}
\end{equation*}
and $w$ is the unique radiating solution to
\begin{equation*}
\Delta^{*}_{\lambda,\mu}w+\rho\omega^{2}w=0 \quad \text { in } \mathbb{R}^2.
\end{equation*}
Thus $w=0$ in $\mathbb{R}^2$, and since the scattering operator is unitary, this shows that $h=S_c^* w^{\infty}=0$.
\end{proof}

\begin{theorem} \label{th1}
Suppose that $\lambda$, $\mu$, $\rho\in L_{+}^{\infty}(\mathbb{R}^2)$ and let $B$, $D \subseteq \mathbb{R}^2$ be open and bounded such that $\mathbb{R}^2\backslash \overline{D}$ is connected. If $B\not\subseteq D$, then for any finite dimensional subspace $V\subseteq [L^2(\mathbb{S})]^{2}$ there exists a sequence $(g_m^{(j)})_{m\in\mathbb{N}}\subseteq V^{\perp}$ such that
\begin{equation*}
\|u_{c,g_m^{(0)}}\|_{[H^{1}(B)]^{2}}\rightarrow \infty \quad {\rm and }  \quad \|u_{c,g_m^{(0)}}\|_{[H^{1}(D)]^{2}} \rightarrow 0 \quad {\rm as}\,\, m \rightarrow \infty;
\end{equation*}
\begin{equation*}
\|u_{c,g_m^{(1)}}\|_{[L^{2}(B)]^{2}}\rightarrow \infty \quad {\rm and }  \quad \|u_{c,g_m^{(1)}}\|_{[H^{1}(D)]^{2}} \rightarrow 0 \quad {\rm as}\,\, m \rightarrow \infty;
\end{equation*}
\begin{equation*}
\|\nabla\cdot u_{c,g_m^{(2)}}\|_{L^{2}(B)}\rightarrow \infty \quad {\rm and }  \quad \|u_{c,g_m^{(2)}}\|_{[H^{1}(D)]^{2}} \rightarrow 0 \quad {\rm as}\,\, m \rightarrow \infty;
\end{equation*}
\begin{equation*}
\|\widehat{\nabla}u_{c,g_m^{(3)}}\|_{[L^{2}(B)]^{2\times2}}\rightarrow \infty \quad {\rm and }  \quad \|u_{c,g_m^{(3)}}\|_{[H^{1}(D)]^{2}} \rightarrow 0 \quad {\rm as}\,\, m \rightarrow \infty
\end{equation*}
where $u_{c,g_m^{(j)}}\in [H_{{\rm loc}}^1(\mathbb{R}^2)]^{2}$ is given by \eqref{ug} with $g=g_m^{(j)}$ ($j=0,1,2,3$).
\end{theorem}
\begin{proof}
Without loss of generality, we assume that $\overline{B} \cap\overline{D}=\emptyset$ and $\mathbb{R}^2\backslash(\overline{B} \cup \overline{D})$ is connected (otherwise we replace $B$ by a sufficiently small ball $\widetilde{B} \subseteq B \backslash \overline{D_{\varepsilon}}$, where $D_{\varepsilon}$ denotes a sufficiently small neighborhood of $D$).

We denote by $P_V: [L^2(\mathbb{S})]^{2} \rightarrow [L^2(\mathbb{S})]^{2}$ the orthogonal projection on $V$. Lemma \ref{lem2} shows that $\mathcal{R}(L_{c,B}^{'}) \cap \mathcal{R}(L_{c,D}')=\mathcal{R}(L_{c,B}^{(j)'}) \cap \mathcal{R}(L_{c,D}')=\{0\}$ $(j=1,2,3)$ and that $\mathcal{R}(L_{c,B}^{'})$, $\mathcal{R}(L_{c,B}^{(j)'})$ are infinite dimensional. Using a simple dimensionality argument (Lemma 4.7 in \cite{Harrach2019}) it follows that (we just show the case $j=1$ for brevity)
\begin{equation*}
\mathcal{R}(L_{c, B}^{(1)'}) \nsubseteq \mathcal{R}(L_{c,D}')+V =\mathcal{R}\left(\left(\begin{array}{ll}
L_{c, D}' & P_V'
\end{array}\right)\right)=\mathcal{R}\left(\binom{L_{c, D}}{P_V}'\right).
\end{equation*}
It then follows from Lemma 4.6 in \cite{Harrach2019} that there is no constant $C>0$ such that
\begin{equation*}
\left\|L_{c,B}^{(1)} g\right\|_{[L^2(B)]^{2}}^2 \leq C^2\left\|\binom{L_{c, D}}{P_V} g\right\|_{[H^1(D)]^{2} \times [L^2(\mathbb{S})]^{2}}^2=C^2\left(\left\|L_{c, D} g\right\|_{[H^1(D)]^{2}}^2+\left\|P_V g\right\|_{[L^2(\mathbb{S})]^{2}}^2\right)
\end{equation*}
for all $g\in [L^2(\mathbb{S})]^{2}$. Hence, there exists a sequence $(\widetilde{g}_m^{(1)})_{m \in \mathbb{N}} \subseteq [L^2(\mathbb{S})]^{2}$ such that
\begin{equation*}
\left\|L_{c, B}^{(1)} \widetilde{g}_m^{(1)}\right\|_{[L^2(B)]^{2}} \rightarrow \infty \quad \text { and } \quad\left\|L_{c, D} \widetilde{g}_m^{(1)}\right\|_{[H^1(D)]^{2}}+\left\|P_V \widetilde{g}_m^{(1)}\right\|_{[L^2(\mathbb{S})]^{2}} \rightarrow 0 \quad \text { as } m \rightarrow \infty.
\end{equation*}
Setting $g_m^{(1)}:=\widetilde{g}_m^{(1)}-P_V \widetilde{g}_m^{(1)} \in V^{\perp}\subseteq[L^2(\mathbb{S})]^{2}$ for any $m \in \mathbb{N}$, we finally obtain
\begin{align*}
\left\|L_{c, B}^{(1)} g_m^{(1)}\right\|_{[L^2(B)]^{2}} \geq\left\|L_{c, B}^{(1)} \widetilde{g}_m^{(1)}\right\|_{[L^2(B)]^{2}}-\left\|L_{c,B}^{(j)}\right\|\left\|P_V \widetilde{g}_m^{(1)}\right\|_{[L^2(\mathbb{S})]^{2}} \rightarrow \infty \quad \text { as } m \rightarrow \infty, \\
\left\|L_{c, D} g_m^{(1)}\right\|_{[H^1(D)]^{2}} \leq\left\|L_{c, D} \widetilde{g}_m^{(1)}\right\|_{[H^1(D)]^{2}}+\left\|L_{c, D}\right\|\left\|P_V \widetilde{g}_m^{(1)}\right\|_{[L^2(\mathbb{S})]^{2}} \rightarrow 0 \quad \text { as } m \rightarrow \infty .
\end{align*}
Substituting the definitions of operators $L_{c,B}^{(1)}$ and $L_{c,D}$, this ends the proof.
\end{proof}

As an application of Theorem \ref{th1} we establish a converse of \eqref{main2} in Theorem \ref{main}.

\begin{theorem}  \label{th3}
Suppose that $\lambda_{j}$, $\mu_{j}$, $\rho_{j}\in L_{+}^{\infty}(\mathbb{R}^2)$ $(j=1,2)$ with $\Omega\subseteq B_R(O)$. If $\mathcal{D} \subseteq \mathbb{R}^2$ is an unbounded domain such that
\begin{equation*}
\lambda_{1}\geq\lambda_{2}, \, \mu_{1}\geq\mu_{2}, \, \rho_{2}\geq \rho_{1}  \quad \text { a.e. in } \mathcal{D},
\end{equation*}
and if $B \subseteq B_R(O)\cap \mathcal{D}$ is open with
\begin{equation} \label{equ1}
\lambda_{1}-\delta_{1}\geq\lambda_{2}, \, \mu_{1}-\delta_{2}\geq\mu_{2}, \, \rho_{2}-\delta_{3} \geq \rho_{1} \quad \text { a.e. in } B \text { for some } \delta_{j}>0,
\end{equation}
then
\begin{equation*}
\Re\left(S_{c_1}^* F_{c_2}\right) \not \leq_{\text {fin }} \Re\left(S_{c_1}^* F_{c_1}\right),
\end{equation*}
i.e., the operator $\Re\left(S_{c_1}^*\left(F_{c_2}-F_{c_1}\right)\right)$ has infinitely many positive eigenvalues. In particular, this implies that $F_{c_1} \neq F_{c_2}$.
\end{theorem}
\begin{proof}
We prove the result by contradiction and assume that
\begin{equation} \label{equ2}
\Re\left(S_{c_1}^*\left(F_{c_2}-F_{c_1}\right)\right) \leq_{\text {fin }} 0 .
\end{equation}
Using the monotonicity relation \eqref{main1} in Theorem \ref{main}, we find that there exists a finite dimensional subspace $V \subseteq [L^2(\mathbb{S})]^{2}$ such that
\begin{equation}  \label{equ3}
\sqrt{8\pi\omega}\,\,\Re\langle S_{c_{1}}^{*}(F_{c_{2}}-F_{c_{1}})g, g\rangle \geq \int_{B_R(O)}E_{\lambda_{1}-\lambda_{2},\mu_{1}-\mu_{2}}(\overline{u}_{c_{1},g},u_{c_{1},g}) +(\rho_{2}-\rho_{1})\omega^{2}|u_{c_{1},g}|^{2}{\rm d}y,
\end{equation}
for all $g\in V^{\bot}$. Combining \eqref{equ1}, \eqref{equ2} and \eqref{equ3}, we obtain that there exists a finite dimensional subspace $\widetilde{V} \subseteq [L^2(\mathbb{S})]^{2}$ such that, for any $g \in \widetilde{V}^{\perp}$,
\begin{align*}
0\geq & \sqrt{8\pi\omega}\,\,\Re\langle S_{c_{1}}^{*}(F_{c_{2}}-F_{c_{1}})g, g\rangle \geq \int_{B_R(O)}E_{\lambda_{1}-\lambda_{2},\mu_{1}-\mu_{2}}(\overline{u}_{c_{1},g},u_{c_{1},g}) +(\rho_{2}-\rho_{1})\omega^{2}|u_{c_{1},g}|^{2}{\rm d}y \\
=&\left(\int_{\mathcal{D}\,\cap B_R(O)}+ \int_{B_R(O) \backslash\overline{\mathcal{D}}}\right) \left(E_{\lambda_{1}-\lambda_{2},\mu_{1}-\mu_{2}}(\overline{u}_{c_{1},g},u_{c_{1},g}) +(\rho_{2}-\rho_{1})\omega^{2}|u_{c_{1},g}|^{2}\right){\rm d}y  \\
\geq& \int_B E_{\delta_{1},\delta_{2}}(\overline{u}_{c_{1},g},u_{c_{1},g}) +\delta_{3}\,\omega^{2}|u_{c_{1},g}|^{2}\mathrm{~d} x \\ &+\int_{B_R(O) \backslash\overline{\mathcal{D}}}E_{\lambda_{1}-\lambda_{2}, \mu_{1}-\mu_{2}} (\overline{u}_{c_{1},g},u_{c_{1},g})+(\rho_{2}-\rho_{1})\omega^{2}|u_{c_{1},g}|^{2}{\rm d}y \\
\geq& \,\delta_{\min}\,C_{1}\,\|u_{c_{1},g}\|_{[H^{1}(B)]^{2}}^{2}-\int_{B_R(O) \backslash\overline{\mathcal{D}}}E_{\hat{\lambda},\hat{\mu}}(\overline{u}_{c_{1},g},u_{c_{1},g}) +\hat{\rho}\,\omega^{2}\,|u_{c_{1},g}|^{2}{\rm d}y \\
\geq& \,\delta_{\min}\,C_{1}\,\|u_{c_{1},g}\|_{[H^{1}(B)]^{2}}^{2} -C_{2}\,\|u_{c_{1},g}\|_{[H^{1}(B_R(O) \backslash\overline{\mathcal{D}})]^{2}}^{2}
\end{align*}
where $C_{1}$, $C_{2}$ are positive constants, $\delta_{\min}:=\min\{\delta_{1},\delta_{2},\delta_{3}\omega^{2}\}$, $\hat{\lambda} =\|\lambda_{1}\|_{L_{+}^{\infty}(\mathbb{R}^2)} +\|\lambda_{2}\|_{L_{+}^{\infty}(\mathbb{R}^2)}$, $\hat{\mu}= \|\mu_{1}\|_{L_{+}^{\infty}(\mathbb{R}^2)}+\|\mu_{2}\|_{L_{+}^{\infty}(\mathbb{R}^2)}$, $\hat{\rho}=\|\rho_{2}\|_{L_{+}^{\infty}(\mathbb{R}^2)}+ \|\rho_{1}\|_{L_{+}^{\infty}(\mathbb{R}^2)}$. However, this contradicts Theorem \ref{th1} with $D=B_R(O) \backslash \overline{\mathcal{D}}$ and $c=c_1$, which guarantees the existence of $\left(g_m\right)_{m \in \mathbb{N}} \subseteq \widetilde{V}^{\perp}$ with
\begin{equation*}
\left\|u_{c_1, g_m}\right\|_{[H^{1}(B)]^{2}}\rightarrow \infty \quad \text { and } \quad \left\|u_{c_1, g_m}\right\|_{[H^{1}(B_R(O) \backslash \overline{\mathcal{D}})]^{2}} \rightarrow 0 \quad \text { as } m \rightarrow \infty .
\end{equation*}
Consequently, $\Re\left(S_{c_1}^*\left(F_{c_2}-F_{c_1}\right)\right)\not\leq_{\text {fin }} 0$.
\end{proof}

\section{Monotonicity based shape reconstruction}

We will consider inhomogeneities in the material parameters of the following type. Let $D_{1}$, $D_{2}$, $D_{3}\subseteq\Omega$ and $D:=D_1 \cup D_2 \cup D_3$. We will now assume that $\lambda$, $\mu$, $\rho\in L_{+}^{\infty}(\mathbb{R}^2)$ are such that
\begin{align} \label{lmr}
&\lambda(x)=\lambda_{0}+\chi_{D_{1}}(x)\psi_{\lambda}(x),\quad \psi_{\lambda}\in L^{\infty}(\Omega),\quad  \psi_{\lambda}(x)>m_{1}, \nonumber\\
&\mu(x)=\mu_{0}+\chi_{D_{2}}(x)\psi_{\mu}(x),\quad  \psi_{\mu}\in L^{\infty}(\Omega),\quad  \psi_{\mu}(x)>m_{2}, \\
&\rho(x)=\rho_{0}-\chi_{D_{3}}(x)\psi_{\rho}(x),\quad  \psi_{\rho}\in L^{\infty}(\Omega),\quad  m_{3}<\psi_{\rho}(x)<M_{3}, \nonumber
\end{align}
where the constants $\lambda_{0}$, $\mu_{0}$, $\rho_{0}>0$ and the bounds $m_{1}, m_{2}, m_{3}> 0$ and $\rho_{0}>M_{3}$. The coefficients $\lambda$, $\mu$ and $\rho$ model inhomogeneities in an otherwise homogeneous background medium given by the coefficients $\lambda_{0}$, $\mu_{0}$ and $\rho_{0}$. In this section, we will give a method to recover ${\rm osupp}(D):={\rm osupp}(\chi_{D})$ (see Section 6 in \cite{Eberle2024}) from the far field operator, and thus the shape of the region where the coefficients differ from the background coefficients $\lambda_{0}$, $\mu_{0}$ and $\rho_{0}$.

Let $B\subseteq\Omega$ be a ball, the test coefficients $\lambda^{\flat}$, $\mu^{\flat}$ and $\rho^{\flat}$ are defined by
\begin{align} \label{lmr2}
&\lambda^{\flat}(x)=\lambda_{0}+\chi_{B}(x)\alpha_{1}, \nonumber\\
&\mu^{\flat}(x)=\mu_{0}+\chi_{B}(x)\alpha_{2},\\
&\rho^{\flat}(x)=\rho_{0}-\chi_{B}(x)\alpha_{3},\nonumber
\end{align}
where $\alpha_{j}\geq0$ $(j=1,2,3)$ are constants.

\begin{theorem}
Let $B \subseteq \Omega$ and $\alpha_j\geq0$ be as in \eqref{lmr2}, and set $\alpha:=(\alpha_1,\alpha_2,\alpha_3)$. The following holds:

(i) Assume that $B \subseteq D_j$, for $j \in \mathbb{I}$ for some $\mathbb{I}\subset \{1,2,3\}$. Then for all $\alpha_j$ with $\alpha_j \leqslant m_j$, $j \in \mathbb{I}$, and $\alpha_j=0$, $j \notin \mathbb{I}$, the operator $\Re\left(S_{c}^*\left(F_{c^{\flat}}-F_{c}\right)\right)$ has finitely many negative eigenvalues.

(ii) If $B \not\subseteq {\rm osupp}(D)$, then for all $\alpha$, $|\alpha| \neq 0$, the operator $\Re\left(S_{c}^*\left(F_{c^{\flat}}-F_{c}\right)\right)$ has infinitely many negative eigenvalues.

Where $F_{c}$ is the far field operator for the coefficients in \eqref{lmr} and $F_{c^{\flat}}$ is the far field operator for the coefficients in \eqref{lmr2}.
\end{theorem}
\begin{proof}
Notice firstly that $\Re\left(S_{c}^*\left(F_{c^{\flat}}-F_{c}\right)\right)$ is a compact self-adjoint operator.

(i) Assume that $B \subseteq D_j$ for $j \in \mathbb{I}$. Choose $0 \leqslant \alpha_j \leqslant m_j$ for $j \in \mathbb{I}$ and $\alpha_j=0$ for $j \notin \mathbb{I}$. Moreover choose $F_{c_{1}}=F_{c}$ and $F_{c_{2}}=F_{c^{\flat}}$ in Theorem \ref{main}. According to Theorem \ref{main} there exists a finite dimensional subspace $V\subseteq [L^2(\mathbb{S})]^{2}$, such that if $g\in V^{\bot}$, then
\begin{align*}
&\sqrt{8\pi\omega}\,\,\Re\langle S_{c}^{*}(F_{c^{\flat}}-F_{c})g, g\rangle \geq \int_{\mathbb{R}^{2}}E_{\lambda-\lambda^{\flat},\mu-\mu^{\flat}}(\overline{u}_{c,g}, u_{c,g}) +(\rho^{\flat}-\rho)\omega^{2}|u_{c,g}|^{2}{\rm d}y \\
&= \int_{\mathbb{R}^{2}} 2(\mu-\mu^{\flat})|\widehat{\nabla} u_{c,g}|^2 +(\lambda-\lambda^{\flat})|\nabla \cdot u_{c,g}|^2 +(\rho^{\flat}-\rho)\omega^{2}|u_{c,g}|^{2}{\rm d}y  \\
&\geq \int_{D_2} 2(m_2-\alpha_2 \chi_B)|\widehat{\nabla} u_{c,g}|^2 \mathrm{~d}y+\int_{D_1}(m_1-\alpha_1 \chi_B)|\nabla \cdot u_{c,g}|^2 \mathrm{~d}y \\
&\quad +\int_{D_3}\omega^2(m_3-\alpha_3\chi_B)|u_{c,g}|^2\mathrm{~d}y \geq0
\end{align*}
where we use the properties in \eqref{lmr} and
\begin{equation*}
E_{\lambda,\mu}(u,v)=2\mu\widehat{\nabla}u: \widehat{\nabla}v+\lambda\nabla\cdot u\, \nabla\cdot v   \quad {\rm with} \quad \widehat{\nabla}u:=\frac{1}{2} \left( \nabla u+(\nabla u)^{\top}\right).
\end{equation*}
That is,
\begin{equation*}
\Re\langle S_{c}^{*}(F_{c^{\flat}}-F_{c})g, g\rangle \geq  0, \quad \forall g \in V^{\bot}.
\end{equation*}
Hence, we have that $\Re\left(S_{c}^{*}(F_{c^{\flat}}-F_{c})\right)$ has finitely many negative eigenvalues.

(ii) We assumed on the contrary that $\Re\left(S_{c}^{*}(F_{c^{\flat}}-F_{c})\right)$ has finitely many negative eigenvalues, then there is a finite dimensional subspace $\widetilde{V} \subseteq [L^2(\mathbb{S})]^{2}$, such that
\begin{equation} \label{ineq}
\Re\langle S_{c}^{*}(F_{c^{\flat}}-F_{c})g, g\rangle \geq  0, \quad \forall g \in \widetilde{V}^{\bot}.
\end{equation}
To obtain a contradiction we consider Theorem \ref{main}, where $F_{c_{1}}=F_{c^{\flat}}$ and $F_{c_{2}}=F_{c}$ and which is rearranged to give
\begin{align*}
&\sqrt{8\pi\omega}\,\,\Re\langle S_{c^{\flat}}^{*}(F_{c^{\flat}}-F_{c})g, g\rangle \leq \int_{\mathbb{R}^{2}}E_{\lambda-\lambda^{\flat},\mu-\mu^{\flat}}(\overline{u}_{c^{\flat},g}, u_{c^{\flat},g}) +(\rho^{\flat}-\rho)\omega^{2}|u_{c^{\flat},g}|^{2}{\rm d}y \\
& =\int_{\mathbb{R}^{2}} 2(\mu-\mu^{\flat})|\widehat{\nabla} u_{c^{\flat},g}|^2 +(\lambda-\lambda^{\flat})|\nabla \cdot u_{c^{\flat},g}|^2+\omega^2(\rho^{\flat}-\rho)|u_{c^{\flat},g}|^2 \mathrm{~d}y \\
&= \int_{\Omega} 2(\psi_\mu \chi_{D_2}-\alpha_2 \chi_{B})|\widehat{\nabla} u_{c^{\flat},g}|^2 +(\psi_\lambda \chi_{D_1}-\alpha_1 \chi_{B})|\nabla \cdot u_{c^{\flat},g}|^2 \\
&\quad +\omega^2(\psi_\rho \chi_{D_3}-\alpha_3 \chi_{B}) |u_{c^{\flat},g}|^2 \mathrm{~d} y.
\end{align*}
By Theorem \ref{th1} we can choose some sequences $(g_m^{(j)})_{m\in\mathbb{N}}\subseteq V^{\perp}$ $(j=1,2,3)$ such that
\begin{equation*}
\|u_{c^{\flat},g_m^{(1)}}\|_{[L^{2}(B)]^{2}}\rightarrow \infty \quad {\rm and }  \quad \|u_{c^{\flat},g_m^{(1)}}\|_{[H^{1}(D)]^{2}} \rightarrow 0 \quad {\rm as}\quad m \rightarrow \infty;
\end{equation*}
\begin{equation*}
\|\nabla\cdot u_{c^{\flat},g_m^{(2)}}\|_{L^{2}(B)}\rightarrow \infty \quad {\rm and }  \quad \|u_{c^{\flat},g_m^{(2)}}\|_{[H^{1}(D)]^{2}} \rightarrow 0 \quad {\rm as}\quad m \rightarrow \infty;
\end{equation*}
\begin{equation*}
\|\widehat{\nabla}u_{c^{\flat},g_m^{(3)}}\|_{[L^{2}(B)]^{2\times2}}\rightarrow \infty \quad {\rm and }  \quad \|u_{c^{\flat},g_m^{(3)}}\|_{[H^{1}(D)]^{2}} \rightarrow 0 \quad {\rm as}\,\, m \rightarrow \infty.
\end{equation*}
Inserting these solutions to the previous inequality yields
\begin{align*}
\sqrt{8\pi\omega}\,\,\Re\langle S_{c^{\flat}}^{*}(F_{c^{\flat}}-F_{c})g_{m}^{(j)}, g_{m}^{(j)} \rangle  \leq &\, C \int_D|\widehat{\nabla} u_{c^{\flat},g_{m}^{(j)}}|^2+|\nabla \cdot u_{c^{\flat},g_{m}^{(j)}}|^2 +|u_{c^{\flat},g_{m}^{(j)}}|^2 \mathrm{~d}y \\
&-\int_{B}\alpha_1|\nabla\cdot u_{c^{\flat},g_{m}^{(j)}}|^2+2\alpha_2|\widehat{\nabla} u_{c^{\flat},g_{m}^{(j)}}|^2+\alpha_3|u_{c^{\flat},g_{m}^{(j)}}|^2 \mathrm{~d}y.
\end{align*}
Since $|\alpha| \neq 0$ and $\alpha_j \geqslant 0$, we see that the last integral becomes large and increasingly negative while the first integral vanishes as $m$ grows, and thus
\begin{equation*}
\Re\langle S_{c}^{*}(F_{c^{\flat}}-F_{c})g_{m}^{(j)}, g_{m}^{(j)}\rangle= \Re\langle S_{c^{\flat}}^{*}(F_{c^{\flat}}-F_{c})g_{m}^{(j)}, g_{m}^{(j)}\rangle <0,
\end{equation*}
for large enough $m$. This is in contradiction with \eqref{ineq} which complete the proof.
\end{proof}

\vspace{0.3cm}
{\bf Acknowledgments.}

The work of J.L. Xiang is supported by the Natural Science Foundation of China (No. 12301542), the Open Research Fund of Hubei Key Laboratory of Mathematical Sciences (Central China Normal University, MPL2025ORG017) and the China Scholarship Council.

\end{document}